\documentclass{article}
\usepackage{amsmath,amsfonts,amsthm,amssymb,color}
\usepackage[dvips]{graphicx}
\usepackage[T1]{fontenc}
\usepackage{enumerate}
\def\authorsPS{}
\newcommand\authPS[8]{\ifnum\authPSc<1\def\authPSc{1}\else, \fi{\large\sffamily #1 #2}%
\edef\authorsPS{\authorsPS \par\vskip 1mm\noindent #1 #2:\hskip 5mm #3, #4, #5, #6, #7, #8}}

\newtheorem{theorem}{Theorem}[section]
\newtheorem{lemma}[theorem]{Lemma}

\theoremstyle{definition}

\theoremstyle{remark}
\newtheorem{remark}[theorem]{Remark}
\newtheorem{example}[theorem]{Example}
\def\authPSc{0}
\newcommand{\beq}[1]{%\marginpar{\footnotesize\sf #1}%
\begin{equation}\label{#1}}
\newcommand{\eeq}{\end{equation}}
\newcommand{\req}[1]{{\rm(\ref{#1})}}
\newcommand{\hten}{{\mathfrak H}}
\newcommand{\law}{\stackrel{\cal L}{\longrightarrow}}
\newcommand{\Nt}{\lfloor nt \rfloor}
\newcommand{\DXj}{\Delta X_{\frac jn}}

\begin{document}

% !!!!!!!!!!!!!!!!!!!!!!!!!!!!!!!!!!!!!!!!!!!!!!!!!!!!!!!!!!!!!!!!!!!!!!!
% !!!                   START HERE                                   !!!
% !!!!!!!!!!!!!!!!!!!!!!!!!!!!!!!!!!!!!!!!!!!!!!!!!!!!!!!!!!!!!!!!!!!!!!!

%% !!  FULL TITLE OF PAPER
% example :

\title{Central limit theorem for functionals of a generalized self-similar Gaussian process}

\author{Daniel Harnett\thanks{
Department of Mathematical Sciences, University of Wisconsin Stevens Point \newline Stevens Point, Wisconsin 54481, dharnett@uwsp.edu}
\, and David Nualart\thanks{  
Department of Mathematics, University of Kansas\newline  
\; 405 Snow Hall, Lawrence, Kansas 66045-2142, nualart@ku.edu \newline D. Nualart is supported by NSF grant DMS1512891 and the ARO grant FED0070445 \newline
}  
}
\date{}
\maketitle

\begin{abstract}
We consider a class of self-similar, continuous Gaussian processes that do not necessarily have stationary increments.  We prove a version of the Breuer-Major theorem for this class, that is, subject to conditions on the covariance function, a generic functional of the process increments converges in law to a Gaussian random variable.  The proof is based on the Fourth Moment Theorem.  We give examples of five non-stationary processes that satisfy these conditions.   
\end{abstract}
\large

\section{Introduction}
We consider a centered Gaussian process $X = \{ X_t, t\ge 0\}$ that is self-similar of order $\beta \in (0,1)$.  That is,  the process $\{a^{-\beta} X_{a t}, t\ge 0 \}$  has the same distribution as the process $X$ for any $a>0$. Consider the function  $\phi: [1,\infty) \rightarrow \mathbb{R}$ given by
\begin{equation} \label{phi}
\phi(x)= \mathbb{E} [X_1X_x].
\end{equation}
This function characterizes the  {\em covariance function}. Indeed, for $0 < s\le t$, we have
\[ R(s,t) = {\mathbb E}\left[ X_s X_t\right] = s^{2\beta}{\mathbb E}\left[ X_1 X_{t/s}\right] = s^{2\beta} \phi\left( \frac ts\right).\]

The best known self-similar Gaussian process is the fractional Brownian motion (fBm), where
\[R(s,t) = \frac 12\left( s^{2H} +t^{2H} - |t-s|^{2H}\right),\]
and the self-similarity exponent $\beta$ is the Hurst parameter  $H \in (0,1)$.  For the fBm, 
\[\phi(x) = \frac 12\left( 1+ x^{2H} - (x-1)^{2H}\right),\; x\ge 1.\]

We will  impose conditions on the function $\phi$ such that
${\mathbb E}\left[ (X_{t+s} - X_t)^2\right] \sim s^{\alpha}$ as $s\to 0$, for some constant $\alpha \in (0,2\beta]$ that we call  the {\em increment exponent}.
For example, $\alpha = 2H$ for the fBm, since ${\mathbb E}\left[ (X_{t+s} - X_t)^2\right] = s^{2H}$. However, there are examples where $\alpha <2\beta$.

Our goal in this paper is to identify a set of conditions on $\alpha, \beta$ and $\phi$, such that we can establish a central limit theorem  for  functionals of the increments of $X$.  More precisely,   let $\gamma= {\cal N}(0,1)$ and consider a function $f\in L^2(\mathbb{R}, \gamma)$,  which has an expansion of the form
\begin{equation}    \label{chaos}
f(x) = \sum_{q=d}^\infty c_q H_q(x), 
\end{equation} 
where $d\ge 1$,  $c_d \not =0$, and   $H_q(x)$ is the $q$th Hermite polynomial, defined as
\[ 
H_q(x) = (-1)^q e^{\frac{x^2}{2}} \frac{d^q}{dx^q} e^{-\frac{x^2}{2}}.
\]
The index $d$ is called the {\it Hermite rank } of $f$.  

For  integers $n\ge 2$ and  $j \ge 0$ define
\[
\DXj = X_{\frac{j+1}{n}} - X_{\frac jn}\qquad\text{ and}\quad Y_{j,n}=  \frac{\DXj}{\| \DXj\|_{L^2(\Omega)}}.
\]
%For a fix any integer $q\ge 2$,
We consider the stochastic process defined by
\begin{equation} \label{hv}
F_n(t) = \frac 1{\sqrt{n}}  \sum_{j=0}^{ \Nt -1} f(  Y_{j,n}), \quad  \Nt \ge 1,
\end{equation}
and $F_n(t)=0$ if $\Nt <1$. 
 It is well known that  if the process $X$ has stationary increments, the convergence to a normal law for the sequence of random variables $F_n(t)$ can be deduced from the following central limit theorem proved by Breuer and Major  in \cite{BreuerMajor}.

\begin{theorem}
Suppose $\{ Y_j,  j\ge 1\}$ is a  centered stationary Gaussian sequence with unit variance, and denote by $\rho(k) =\mathbb{E}(Y_n Y_{n+k})$ its covariance function.  Consider a function $f\in L^2(\mathbb{R}, \gamma)$  with   Hermite rank $d\ge 1$.  Then the functional
\beq{Breuer_Major_classic} V_n = \frac{1}{\sqrt n} \sum_{j=1}^n f(Y_j)\eeq
converges in distribution to a  normal  law ${\cal N}(0, \sigma^2)$ as $n$ tends to infinity, provided $\sum_{k\in \mathbb{Z}} |\rho(k)|^d <\infty$, and  in this case $\sigma^2 =  \sum_{q=d}^\infty  q! c_q^2\sum_{k\in \mathbb{Z}} \rho(k)^q$.
\end{theorem}

For example, in the case of the fBm,  the sequence $\{Y_{j,n},    0\le j \le n-1\}$ defined by
\[
Y_{j,n}= n^{H}\left(B^H_{\frac{j+1}{n}} - B^H_{\frac jn}\right),
\]
is stationary, and  the Breuer-Major theorem implies that if $d\ge 2$   and $H < 1-\frac{1}{2d}$, the sequence of random variables
\begin{equation} 
 \frac 1{\sqrt{n}}  \sum_{j=0}^{n-1} f\left( n^H(B^H_{\frac{j+1}{n}}-B^H_{\frac jn})\right)
\end{equation}
converges in law as $n$ tends to infinity to a Gaussian random variable, with mean zero and variance given by
\beq{Variance_fBm} \sigma^2 =  \sum_{q=d} ^\infty c_q^2 \frac{q!}{2^q} \sum_{m\in{\mathbb Z}} \left( |m+1|^{2H}-2|m|^{2H}+|m-1|^{2H}\right)^q.\eeq
See \cite{BreuerMajor} and Theorem 7.4.1 of \cite{NoP11}.

 Our main result (see Theorem  3.4) says that if   the covariance of $X$ satisfies certain conditions and  the increment exponent $\alpha$ satisfies $0< \alpha < 2-\frac 1d$, then the finite dimensional distributions of processes $\{F_n(t), t\ge 0\}$ defined in 
 (\ref{hv}) converge in law to those of  a Brownian motion with scaling given by 
 (\ref{Variance_fBm}), where $2H$ is replaced by $\alpha$.  Notice that the sequence  of scaled increments  $\{Y_{j,n},   j\ge 0\}$ is not necessarily stationary and we cannot deduce this result form the Breuer-Major theorem.  On the other hand, the relevant parameter in the limit theorem is the increment exponent $\alpha$, instead of the self-similarity parameter $\beta$.
 
The convergence in law of the finite-dimensional distributions in Theorem 3.4 follows from the Fourth Moment Theorem \cite{NOrtiz, NP05}, which represents a drastic simplification of the method of moments to show the convergence to a normal distribution.   In order to establish convergence, it is sufficient to show the convergence of the variances,  and that a condition involving the contraction operator  is satisfied. In this paper, we use extended versions of the Fourth Moment Theorem \cite{HuNu, PT}.

For the particular case of a single Hermite polynomial, that is $f = H_q$ for some $q\ge 2$, one can show  the convergence in total variation of the marginal distributions and a functional central limit theorem (see Section 3.4).

We show examples of known processes that satisfy the required conditions, including:
\begin{enumerate}[(a)]
\item  bifractional Brownian motion (see \cite{Houdre}),
\item subfractional Brownian motion (see \cite{Bojdecki}), 
\item an `arcsine' Gaussian process introduced in a paper by Jason Swanson \cite{Swanson07}, and 
\item  two self-similar Gaussian processes that form the decomposition of a process discussed in a paper by Durieu and Wang \cite{DuWa}.  \end{enumerate}
In examples (a), (b) and (d), the self-similarity and increment exponents are the same, that is, $\alpha = 2\beta$.  This is not true in example (c),  where $\alpha < 2\beta$.

The outline of this paper is as follows.  In Section 2, we give definitions and background needed to use the Fourth Moment Theorem. In Section 3, we introduce a set of covariance conditions on $X$, and  we state Theorem \ref{thm3}, which shows that the finite dimensional distributions of $F_n$ converge in law when these conditions are met.  Some particular applications of Theorem \ref{thm3} are discussed.  Section 4 discusses the examples (a) - (d) above, and Section 5 contains the proofs of two technical lemmas.

\subsection{Acknowledgement}  The authors wish to thank an anonymous referee for helpful comments, especially regarding the proof of Theorem \ref{thm3}.

%%%%%%%%%%%%%%%%%%%%%%%%%%Section 2
%%%%%%%%%%%%%%%%%%%%%%%%%%
\section{Theoretical background}
Following is a brief description of some identities that will be used.  The reader may refer to \cite{NoP11, Nualart} for detailed coverage of this topic.  
Let $Z = \{ Z(h), h\in\cal{H}\}$ be an {\em isonormal Gaussian process} on a probability space $( \Omega, {\cal F}, P )$,  indexed by a real separable Hilbert space $\cal{H}$.  That is, $Z$ is a family of Gaussian random variables such that ${\mathbb E}[ Z(h)] =0$ and ${\mathbb E}\left[Z(h)Z(g)\right] = \left< h,g\right>_{\cal{H}}$ for all $h,g \in\cal{H}$.  We will always assume that $\cal F$ is the filtration generated by $Z$. 
  
For integers $q \ge 1$, let ${\cal H}^{\otimes q}$ denote the $q$th tensor product of ${\cal H}$, and let ${\cal H}^{\odot q}$ denote the subspace of symmetric elements of ${\cal H}^{\otimes q}$.

Let $\{ e_n, n\ge 1\}$ be a complete orthormal system in ${\cal H}$.  For functions $f, g \in {\cal H}^{\odot q}$ and $p\in\{1, \dots, q\}$, we define the $p$th-order contraction of $f$ and $g$ as that element of ${\cal H}^{\otimes 2(q-p)}$ given by
\beq{contract} f\otimes_p g = \sum_{i_1, \dots , i_p=1}^\infty \left< f, e_{i_1} \otimes \cdots\otimes e_{i_p}\right>_{{\cal H}^{\otimes p}} \otimes \left< g, e_{i_1} \otimes \cdots\otimes e_{i_p}\right>_{{\cal H}^{\otimes p}},\eeq
where $f\otimes_0 g = f\otimes g$ by definition and, if $f,g \in {\cal H}^{\odot q}$,  $f\otimes_q g = \left< f,g\right>_{{\cal H}^{\otimes q}}$.  In particular, if $f,g$ are symmetric functions in ${\cal H}^{\otimes 2} = L^2({\mathbb R}^2, {\cal B}^2, \mu^2)$ for a  measure $\mu$, then we have
\beq{contract_integrl} f \otimes_1 g = \int_{\mathbb R} f(s, t_1) g(s, t_2)~\mu(ds).\eeq

Let ${\cal H}_q$ be the $q$th Wiener chaos of $Z$, that is, the closed linear subspace of $L^2(\Omega)$ generated by the random variables $\{ H_q(Z(h)), h \in {\cal H}, \|h \|_{\cal H} = 1 \}$, where $H_q(x)$ is the $q$th Hermite polynomial.  It can be shown (see \cite{NoP11}, Proposition 2.2.1) that if $Z, Y \sim {\cal N}(0,1)$ are jointly Gaussian, then
\begin{equation}  
\label{Herm_cov} 
{\mathbb E}\left[ H_p(Z) H_q(Y)\right] = 
\begin{cases}
p!\left({\mathbb E}\left[ ZY\right]\right)^p & {\rm if}\; p=q\\0&{\rm otherwise}
\end{cases}.
\end{equation}
For $q \ge 1$, it is known that the map 
\beq{Hmap} I_q(h^{\otimes q}) = H_q(Z(h))\eeq
provides a linear isometry between ${\cal H}^{\odot q}$ (equipped with the modified norm $\sqrt{q!}\| \cdot\|_{{\cal H}^{\otimes q}}$) and ${\cal H}_q$. The random variable $I_q(\cdot)$ is the generalized Wiener-It\^o stochastic integral (see \cite{NoP11}, Theorem 2.2.7).  By convention, ${\cal H}_0 = \mathbb{R}$ and $I_0(x) = x$.  

It is well known that $L^2(\Omega)$ can be decomposed into an orthogonal sum of the spaces ${\cal H}_q$.  Hence, any $F\in L^2(\Omega)$ has a {\em Wiener chaos expansion}
\beq{W_chaos} F = \sum_{q=0}^\infty I_q (f_q),\eeq
where $f_0 = {\mathbb E}[F]$ and the $f_q \in {\cal H}^{\odot q}$, $q\ge 1$ are uniquely determined by $F$ (see Theorem 1.1.2 of \cite{Nualart}).   

 The purpose of the above discussion is to provide sufficient background to use the Fourth Moment Theorem.   This theorem, first published in 2005, has inspired an extensive body of literature, and provided solution techniques to a new class of problems.  This first version of the theorem was proved in \cite{NP05}.  Since then, other equivalent conditions have been added  \cite{NoP11, NOrtiz}.  A key advantage of this theorem is that, unlike the method of moments, it is sufficient to check the convergence of the moments of up to order four.   

\begin{theorem}[Fourth Moment Theorem]  \label{thm1} Fix an integer $q \ge 2$.  For integers $n \ge 1$, let $F_n = I_q(f_n)$ be a sequence of random variables belonging to the $q$th Wiener chaos of $X$, so that $f_n \in {\cal H}^{\odot q}$.  Assume that ${\mathbb E}\left[ F_n^2\right] \longrightarrow \sigma^2 \ge 0$ as $n\to\infty$.  
Then the following are equivalent:
\begin{enumerate}[(a)]
\item As $n \to \infty$, $F_n$ converges in distribution to $N \sim {\cal N}(0,\sigma^2)$.
\item  $\lim_{n\to\infty}{\mathbb E}\left[ F_n^4\right] = 3\sigma^4 = {\mathbb E}\left[ N^4\right]$.
\item  For each integer $1\le r \le q-1$,
 $\lim_{n \to \infty} \left\| f_n \otimes_r f_n \right\|_{{\cal H}^{\otimes 2(q-r)}} =0$.
\end{enumerate}\end{theorem}

We have this multidimensional extension due to Peccati and Tudor:
\begin{theorem}[\cite{PT}]\label{Vec4MT}For an integer $k \ge 1$,  let $F_n = (F_n^1, \dots, F_n^k)$ be a sequence of random vectors, and let $1\le q_1\le \dots \le q_k <\infty$ be integers such that for each $j = 1, \dots, k$, $F_n^j = I_{q_j}(f^j_n)$ for some kernel $f_n^j\in {\cal H}^{\odot q_j}$.  Moreover, assume that each $\mathbb{E}\left[ (F_n^j)^2\right] \longrightarrow \sigma_j^2\ge 0$ as $n$ tends to infinity, and that $\lim_{n\to\infty} \mathbb{E}\left[ F_n^j F_n^\ell\right] = 0$ for all $j \neq \ell$.  Then the following are equivalent:
\begin{enumerate}[(a)]
\item  $F_n$ converges in distribution as $n \to \infty$ to ${\cal N}(0, \mathbf{\Sigma})$, where $\mathbf{\Sigma}$ is the $k\times k$ diagonal matrix with entries $\sigma_1^2, \dots, \sigma_k^2$.
\item  For each $j = 1, \dots, k$, $F_n^j$ converges in distribution as $n\to\infty$ to ${\cal N}(0,\sigma_j^2)$.
\end{enumerate}
\end{theorem}
 
 For a general sequence of square integrable random variables, the convergence to the normal distribution can be deduced from  the Wiener chaos expansion. Using Theorem \ref{Vec4MT}, one can show that if every projection on the Wiener chaos satisfies the hypotheses of the Fourth Moment Theorem, then their limits are independent and a central limit theorem holds for the global sequence. This phenomenon can be described as a {\it chaotic central limit theorem}.  
   
\begin{theorem}[\cite{HuNu}]\label{thm2} Let $\{ F_n\}$ be a sequence in $L^2(\Omega)$ such that ${\mathbb E}\left[ F_n\right]=0$ for all $n$.  Write each $F_n$ in the form 
\[ F_n = \sum_{q=1}^\infty I_q(f_{n,q}),\]
and suppose that the following conditions hold:
\begin{enumerate}[(a)]
\item  For each $q\ge 1$, 
$\lim_{n\to\infty} q! \left\| f_{n,q}\right\|^2_{{\cal H}^{\otimes q}} = \sigma^2_q$, for some $\sigma^2_q \ge 0$.
\item  $\sigma^2 = \sum_{q=1}^\infty \sigma_q^2 < \infty$.
\item  For each $q\ge 2$ and $r= 1, \dots, q-1$,
$ \lim_{n\to\infty} \left\| f_{n,q} \otimes_r f_{n,q}\right\|^2_{{\cal H}^{\otimes 2(q-r)}} =0.$
\item  $ \lim_{N \to \infty} \sup_{n\ge 1} \left(\sum_{q=N+1}^\infty q! \left\| f_{n,q}\right\|^2_{{\cal H}^{\otimes q}}\right) =0$.
\end{enumerate}
Then as $n$ tends to infinity, $F_n \law {\cal N}(0,\sigma^2)$.
\end{theorem}

\medskip
In the following sections, the symbol $C$ denotes a generic positive constant, which may change from line to line.  The value of $C$ might depend on $T$ and the properties of the process $X$.

%%%%%%%%%%%%%%%%%%%%%%%   Section 3
\section{Central limit theorem for variations of a self-similar Gaussian process}%%%%%%%%%%%%%%%%%%%%%
\subsection{Defining characteristics of the process} 
Let $X = \{X_t, t\ge 0\}$ denote a centered self-similar Gaussian process with self-similarity parameter $\beta \in(0,1)$.     We introduce the following conditions on the function $\phi$ defined in (\ref{phi}), where   $\alpha \in (0 ,  2\beta]$:

\begin{enumerate}[(H.1)]
\item  $\phi$ has the form $\phi(x) = -\lambda (x-1)^\alpha + \psi(x)$, where $\lambda>0$, $\psi(x)$ is twice-differentiable on an open set containing $[1,\infty)$, and there is a constant $C \ge 0$ such that for any $x\in (1,\infty)$
	\begin{enumerate}[(i)]
	\item  $|\psi'(x)| \le Cx^{\alpha -1}$;
	\item $|\psi''(x)| \le Cx^{-1} (x-1)^{\alpha-1}$; and
	\item $\psi'(1)= \beta \psi(1)$, when $\alpha \ge 1$.
	\end{enumerate}
	\end{enumerate}

The lemma below shows that   $\alpha$  satisfies
${\mathbb E}\left[ (X_{t+s} - X_t)^2\right] \sim s^{\alpha}$,
and we call this value the 
  {\em increment exponent}.

\begin{lemma}  \label{lem3.1}Under (H.1)  for $0<s\le t$ we have
\[
{\mathbb E}\left[ (X_{t+s} - X_t)^2\right] = 2\lambda t^{2\beta -\alpha}s^\alpha + g_1(t,s),
\]
where $|g_1(t,s)| \le Cst^{2\beta-1}$ if $\alpha<1$ and $|g_1(t,s)| \le Cs^2t^{2\beta-2}$ if $\alpha\ge 1$, for some constant $C$.
\end{lemma}

\begin{proof}
We can write
\begin{align*}
{\mathbb E}\left[ (X_{t+s}-X_t)^2\right] &= (t+s)^{2\beta}\phi(1) + t^{2\beta}\phi(1) - 2t^{2\beta}\phi\left( 1+ \frac st\right)\\
&= \phi(1)\left( (t+s)^{2\beta} - t^{2\beta}\right) + 2t^{2\beta}\left( \phi(1) - \phi\left( 1+\frac st\right)\right)\\
&=\phi(1)\left( (t+s)^{2\beta} - t^{2\beta}\right) +2t^{2\beta}\left( \psi(1) + \lambda\left( \frac st\right)^\alpha - \psi\left( 1+ \frac st\right)\right)\\
&= 2\lambda t^{2\beta-\alpha} s^\alpha + \psi(1)\left( (t+s)^{2\beta} - t^{2\beta}\right) - 2t^{2\beta}\int_1^{1+\frac st} \psi'(y)~dy \\
&= 2\lambda t^{2\beta-\alpha}s^\alpha + g_1(t,s).
\end{align*}
If $\alpha<1$,  it follows from the Mean Value Theorem and the estimate on $|\psi'(x)|$ that  $| g_1(t,s) | \le C st^{2\beta -1}$. If $\alpha \ge 1$ we use property (iii) to get
\[
g_1(t,s)=2 t^{2\beta} \int_1^{1+\frac  st}  [\psi'(1) y^{2\beta -1} - \psi'(y) ]dy,
\]
which implies that   $| g_1(t,s) | \le C s^2t^{2\beta -2}$ due to the estimate on  $|\psi''(x)|$.
\end{proof}

Notice that Lemma  \ref{lem3.1} implies that for $0\le s < t$,
\begin{equation} \label{eq1}
|g_1(t,s)| \le  C s^{\alpha +\varepsilon} t^{2\beta-\alpha -\varepsilon},
\end{equation}
for any $\varepsilon>0$ such that  $1-\alpha -\varepsilon >0$ if $\alpha <1$ and $2-\alpha -\varepsilon >0$ if $1\le\alpha <2$.

\begin{lemma}  \label{lem3.2} Assume condition  (H.1). Then,  
\begin{enumerate}[(a)]
\item  For $0<2s \le t$,  we have  
\[
{\mathbb E}\left[ (X_{t+s} - X_t)(X_t - X_{t-s})\right] = (2^\alpha-2)\lambda t^{2\beta -\alpha}s^\alpha + g_2(t,s),
\]
where $|g_2(t,s) | \le C s^2(t-s)^{2\beta-2}+Cs^{\alpha+1}(t-s)^{2\beta-\alpha-1}$.
\item  For $0<2s\le \frac t3 \le r \le t-2s$, 
\begin{multline*}
{\mathbb E}\left[ (X_t - X_{t-s})(X_r - X_{r-s})\right]\\ = \lambda(r-s)^{2\beta-\alpha}\left[  (t-r-s)^\alpha +(t-r+s)^\alpha- 2(t-r)^\alpha \right] + g_3(r,t,s),
\end{multline*}
where 
 $\left| g_3(r,t,s)\right| \le Cs^2(r-s)^{2\beta - \alpha -1}(t-r-s)^{\alpha -1}+Cs^2(r-s)^{2\beta - 2}$.
\end{enumerate}\end{lemma}

\begin{proof}

\medskip
For (a), we can write ${\mathbb E}\left[ (X_{t+s} - X_t)(X_t - X_{t-s})\right]$ as 
\begin{align*}
 &{\mathbb E}\left[ X_{t+s}X_t  -X_t^2- X_{t+s}X_{t-s}+ X_tX_{t-s}\right]\\
&\qquad= t^{2\beta}\left( \phi\left(1+\frac st\right) -\phi(1)\right) - (t-s)^{2\beta}\left( \phi\left( 1+\frac{2s}{t-s}\right)-\phi\left( 1+\frac{s}{t-s}\right)\right)\\
&\qquad=-\lambda t^{2\beta-\alpha}s^\alpha +(t-s)^{2\beta-\alpha}\left( \lambda (2s)^\alpha -\lambda s^\alpha\right)\\
&\qquad\qquad\quad + t^{2\beta}\int_0^{\frac st} \psi'(1+y)~dy - (t-s)^{2\beta}\int_0^{\frac{s}{t-s}}\psi'\left( 1+\frac{s}{t-s} +y\right)~dy\\
&\qquad=\lambda(2^\alpha-2)t^{2\beta-\alpha}s^\alpha + (1-2^\alpha)\lambda s^\alpha\left( t^{2\beta-\alpha}-(t-s)^{2\beta-\alpha}\right)\\
&\qquad\quad\qquad+\left( t^{2\beta}-(t-s)^{2\beta}\right)\int_0^{\frac st} \psi'(1+y)~dy\\
&\qquad\quad\qquad +(t-s)^{2\beta}\int_0^{\frac st} \psi'(1+y) - \psi'\left( 1+\frac{s}{t-s} +y\right)~dy\\
&\qquad\quad\qquad -(t-s)^{2\beta}\int_{\frac st}^{\frac{s}{t-s}} \psi'\left( 1+\frac{s}{t-s} +y\right)~dy\\
&\qquad= (2^\alpha-2)\lambda s^\alpha t^{2\beta-\alpha} + g_2(t,s),
\end{align*}
where, given conditions on $\psi$ and its derivatives, we have that  \[|g_2(t,s)| \le Cs^2(t-s)^{2\beta -2}+ Cs^{\alpha +1}(t-s)^{2\beta-\alpha-1}.\]

Next, for (b),
\begin{align*}
&{\mathbb E}\left[ (X_t - X_{t-s})(X_r - X_{r-s})\right]\\
&\qquad= r^{2\beta}\left( \phi\left(\frac tr\right) - \phi\left(\frac{t-s}{r}\right)\right) - (r-s)^{2\beta}\left( \phi\left(\frac{t}{r-s}\right) - \phi\left( \frac{t-s}{r-s}\right)\right)\\
&\qquad=\lambda r^{2\beta-\alpha}\left(   (t-r-s)^\alpha -(t-r)^{\alpha} \right) -\lambda(r-s)^{2\beta-\alpha}\left(   (t-r)^\alpha -(t-r+s)^\alpha \right)\\
&\qquad\qquad\quad +r^{2\beta}\int_0^{\frac sr} \psi'\left( \frac{t-s}{r} +y\right)~dy - (r-s)^{2\beta}\int_0^{\frac{s}{r-s}} \psi'\left( \frac{t-s}{r-s} +y\right)~dy\\
&\qquad=\lambda(r-s)^{2\beta-\alpha}\left(  (t-r-s)^\alpha +   (t-r+s)^\alpha- 2(t-r)^\alpha \right)\\
&\qquad\qquad\quad -\lambda\left( r^{2\beta-\alpha} - (r-s)^{2\beta-\alpha}\right)\left( (t-r)^\alpha - (t-r-s)^{\alpha}\right)\\
&\qquad\qquad\quad +r^{2\beta}\int_0^{\frac sr} \psi'\left( \frac{t-s}{r} +y\right)~dy - (r-s)^{2\beta}\int_0^{\frac{s}{r-s}} \psi'\left( \frac{t-s}{r-s} +y\right)~dy\\
&\qquad= \lambda(r-s)^{2\beta-\alpha} \left(  (t-r-s)^\alpha +   (t-r+s)^\alpha- 2(t-r)^\alpha \right) + g_3(r,t,s),
\end{align*}
where  
\begin{align*}
&g_3(r,t,s)= \lambda\left(  (r-s)^{2\beta-\alpha}- r^{2\beta-\alpha} \right)\left( (t-r)^\alpha - (t-r-s)^{\alpha}\right) \\
& \qquad \quad \qquad + \left( r^{2\beta}-(r-s)^{2\beta}\right)\int_0^{\frac{s}{r-s}} \psi'\left( \frac{t-s}{r-s} +y\right)~dy\\
&\qquad\quad\qquad - r^{2\beta}\int_0^{\frac{s}{r}} \left[  \psi'\left( \frac{t-s}{r-s}+y\right) 
  - \psi'\left( \frac{t-s}{r}+y\right) \right]dy\\
  & \qquad \quad \qquad  -r^{2\beta} \int_{\frac{s}{r}}^{\frac{s}{r-s}} \psi'\left( \frac{t-s}{r} +y\right)~dy.
\end{align*}
Using  the Mean Value Theorem, the fact that  $\frac r2 \le r-s \le r$  and bounds on the derivatives of $|\psi|$, we have that, for a constant $C$,
\begin{align*}
| g_3(r,t,s)| &\le Cs^2(r-s)^{2\beta-\alpha -1}(t-r-s)^{\alpha-1}+Cs^2(r-s)^{2\beta-2}\left(\frac{t-s}{r-s}\right)^{\alpha-1}\\
&\qquad\qquad\quad+Cs^2(r-s)^{2\beta-1}\left(\frac{t-s}{r-s}\right)^{-1}\left(\frac{t-r}{r-s}\right)^{\alpha-1}\left(\frac{t-s}{r(r-s)}\right)\\
&\qquad\qquad\quad+Cs^2(r-s)^{2\beta-2}\left(\frac{t-s}{r-s}\right)^{\alpha-1}.
\end{align*}
If $\alpha <1$, then $(t-s)^{\alpha-1} \le (t-r-s)^{\alpha-1}$ and $(t-r)^{\alpha-1} \le (t-r-s)^{\alpha-1}$. Therefore,  we have the bound
\[
| g_3(r,t,s)| \le Cs^2(r-s)^{2\beta-\alpha -1}(t-r-s)^{\alpha-1}.\]
In the case $\alpha \ge 1$, then $r \ge t/3$ implies $(t-s)/(r-s) \le 5$. Using this inequality and $t-r\le 2(t-r-s)$,  we obtain
\[
| g_3(r,t,s)| \le Cs^2(r-s)^{2\beta-\alpha -1}(t-r-s)^{\alpha-1} + Cs^2(r-s)^{2\beta-2}.
\]
Then, the proof of  part (b) is complete.

\end{proof}

\begin{example}
Let $B^H = \{ B^H_t, t\ge 0\}$ denote a fractional Brownian motion with Hurst parameter $H$. Then condition (H.1) is satisfied with
 $\alpha = 2\beta = 2H$, $\lambda = \frac 12$. In this case, we obtain
\begin{align*}
{\mathbb E}\left[ (B^H_{t+s} - B^H_t)^2\right] &= s^{2H}\\
{\mathbb E}\left[ (B^H_{t+s} - B^H_{t})(B^H_t - B^H_{t-s})\right]&=\frac 12 (2^{2H}-2)s^{2H}\\
{\mathbb E}\left[ (B^H_t - B^H_{t-s})(B^H_r - B^H_{r-s})\right]&= \frac 12 \left(   (t-r+s)^{2H} +(t-r-s)^{2H}-2(t-r)^{2H} \right),\end{align*}
which means that  $g_1 = g_2 = g_3 =0$ in  Lemmas 3.1 and 3.2.  This means that, in the general case, 
 we can think of $X$ as a process that is similar to the fBm, but with an additional, lower-order correction term on the covariance.   In Section 4 we give examples where the terms $g_i$, $i=1,2,3$, are nonzero.
\end{example}

\medskip
We will make use of the following
 additional condition on the behavior of the first to derivatives of $\phi$ at infinity, which cannot be deduced from condition (H.1).
\begin{enumerate}[(H.2)]
\item  There are constants $C>0$ and $1 < \nu \le 2$ such that for all $x \ge  2$,
	\begin{enumerate}[(i)]
	\item  $	|\phi'(x)| \le \begin{cases} C(x-1)^{-\nu}&\text{ if  }\alpha < 1\\C(x-1)^{\alpha -2}&\text{ if  }\alpha \ge 1,\end{cases}$
	\item $|\phi''(x)| \le \begin{cases} C(x-1)^{-\nu-1}&\text{ if  }\alpha < 1\\C(x-1)^{\alpha -3}&\text{ if  }\alpha \ge 1.\end{cases}$
	\end{enumerate}
 
\end{enumerate}

\medskip
\subsection{Central limit theorem}
We are now ready to state the main result.  
 
\begin{theorem}   \label{thm3} Suppose a self-similar Gaussian process $(X,\phi)$ satisfies (H.1) and (H.2) above.  For $n \ge 2$, consider the stochastic process defined in (\ref{hv}), that is, 
\[
F_n(t) = \frac 1{\sqrt{n}}  \sum_{j=0}^{ \Nt -1} f(  Y_{j,n}), \quad \Nt \ge 1,
\]
and  $F_n(t)=0$  if $\Nt <1$. We assume that $f\in L^2(\mathbb{R}, \gamma)$ has the expansion (\ref{chaos}) with Hermite rank $d\ge 2$.
 Then,  if $\alpha  < 2-\frac 1d$,   the finite dimensional distributions of the  processes $\{F_n, n\ge 2\}$ converge in law  to those of a Brownian motion with scaling given by  $\sigma^2= \sum_{q=d} ^\infty c_q^2 \sigma^2_q$, where 
\beq{var_main}  \sigma_q^2 =2^{-q} q!  
\sum_{m\in \mathbb{Z}} \left(   |m+1|^\alpha +|m-1|^\alpha- 2|m|^\alpha \right)^q.
\eeq
\end{theorem}

\subsection{Proof of Theorem \ref{thm3}} %%%%%%%%%%%%%%%%%%%%%%%
We show the convergence of the finite dimensional distributions  using Theorem \ref{Vec4MT} and Theorem \ref{thm2}.  For an integer $p\ge 2$, choose times $0< t_1 < \dots < t_p < \infty$, and define $t_0 = 0$.  For $i = 1, \dots, p$, define $G_n(t_i) = F_n(t_i) - F_n(t_{i-1})$.  We   want to show that each $G_n(t_i)$ converges in law to ${\cal N}(0,\sigma^2(t_i - t_{i-1}))$ applying Theorem \ref{thm2}.  Without loss of generality, it is enough to prove this for $G_n(t_1) = F_n(t_1)$.  
\medskip
  
The projection of $F_n(t_1)$ on the  Wiener chaos of order $q$ is 
\[
 F_{n,q}(t_1)=\frac {c_q} {\sqrt{n}} \sum_{j=0}^{\lfloor nt_1 \rfloor -1} H_q(Y_{j,n}).
 \]
  By \req{Hmap}, we can write $F_{n,q}$ in terms of the stochastic integral operator $I_q$
\[ 
F_{n,q}(t_1)  =  \frac {c_q} {\sqrt{n}}
\sum_{j=0}^{\lfloor nt_1 \rfloor -1}  \xi_{j,n} ^{-q} I_q \left( \partial_{\frac jn}^{\otimes q}\right),\]
where  we use the notation 
\[
 \xi_{j,n}= \| \Delta X_{\frac jn}\|_{L^2(\Omega)}.
 \]
The symbol $\partial_{j/n}$ denotes the indicator function of the interval $[\frac jn, \frac{j+1}{n}]$, that is
$\partial_{\frac jn} = {\mathbf 1}_{[\frac jn, \frac{j+1}{n}]}$. 
To verify the   conditions of Theorem \ref{thm2}, we adopt the following  Hilbert space notation.  The indicator function ${\mathbf 1}_{[0,t]}$ is an element of the Hilbert space $\hten $,  defined as the closure of the set of step functions with respect to the inner product
\[
 \left< {\mathbf 1}_{[0,s]}, {\mathbf 1}_{[0,t ]}\right>_\hten = {\mathbb E}\left[ X_{s}X_{t}\right], \quad s,t \ge 0.
 \] 
With this representation, we can write $F_{n,q}(t_1) = I_q (f_{n,q})$, where
\begin{equation}
\label{eq5} f_{n,q} = f_{n,q}(t_1)  =  \frac {c_q} {\sqrt{n}}\sum_{j=0}^{\lfloor nt_1 \rfloor -1}  h_{j,n}   ^{\otimes q},
\end{equation}
where $h_{j,n}= \xi_{j,n} ^{-q} \partial_{\frac jn}$.
It is apparent that $f_n \in \hten^{\odot q}$. 

Now we proceed to verify the conditions of Theorem \ref{thm2}.

\medskip
\noindent
{\it  Proof of condition (a):}
We want to show that for any $q\ge d$,  ${\mathbb E}\left[F_{n,q}(t_1)^2\right]$ converges to $\sigma_q^2t_1$, where $\sigma_q^2$ is given by \req{var_main}.  
This is the contents of Lemma \ref{lem5.2},  whose proof is given in Section 5.
 
\medskip
\noindent
{\it  Proof of condition (b):}  This is obvious by definition.
 
  \medskip
  \noindent
{\it  Proof of condition (c):}  
We wish to show that for each $r=1,2,\dots, q-1$,
\beq{Contraction_cond} \lim_{n\to\infty}\left\| f_{n,q} \otimes_r f_{n,q}\right\|^2_{\hten^{\otimes 2(q-r)}} = 0.\eeq
By \req{contract_integrl} we have for each $r = 1, \dots, q-1$
\[
 f_{n,q} \otimes_r f_{n,q} =   \frac {c_q^2} n \sum_{j,k=0}^{\lfloor nt_1 \rfloor -1}
 \left<h_{j,n} ,  h_{k,n}\right>_\hten^r \left( h_{j,n}^{\otimes q-r} \otimes  h_{k,n}^{\otimes q-r}\right),
 \]
which is an element of $\hten^{\otimes 2(q-r)}$.  It follows that
\begin{equation}
\left\| f_{n,q} \otimes_r f_{n,q} \right\|^2_{\hten^{\otimes 2(q-r)}}
= \frac {c_q^4} {n^2} \sum_{j,k,\ell,m=0}^{\lfloor nt_1 \rfloor -1}
\left<  h_{j,n},  h_{k,n}\right>_\hten^r 
\left< h_{\ell,n},  h_{m,n}\right>_\hten^r\left< h_{j,n},  h_{\ell,n}\right>_\hten^{q-r}
\left<  h_{k,n},  h_{m,n}\right>_\hten^{q-r}.
\label{Contraction_stmt}
\end{equation}

\medskip
We proceed in a manner similar to the  proof of  the convergence (1.3) in \cite{DaNoNu}. Fix an integer $M\ge 1$.  First we decompose the set of multi-indexes $D=[0,\lfloor nt_1 \rfloor -1]^4$ into $D=D_{1,M}\cup D_{2,M}$, where
\[
D_{1,M}= \{(j,k,\ell,m)\in D:  |j-k|\le M, |\ell-m| \le M, |j-\ell| \le M\},
\]
\[
D_{2,M}=D_{3,M} \cup D_{4,M} \cup D_{5,M},
\]
\[
D_{3,M}=  \{(j,k,\ell,m)\in D:  |j-k| > M\}, 
\]
\[  D_{4,M}=  \{(j,k,\ell,m)\in D:  |\ell -m| > M\}, 
\]
and
\[
D_{5,M}=  \{(j,k,\ell,m)\in D:  |j-\ell| > M\}.
\]
Taking into account that $\| h_{j,n}\|_{\hten}=1$ and using Cauchy-Schwarz inequality, it follows that
\[
 \frac {1} {n^2} \sum_{(j,k,\ell,m)\in D_{1,M}}
\left| \left<  h_{j,n},  h_{k,n}\right>_\hten^r 
\left< h_{\ell,n},  h_{m,n}\right>_\hten^r\left< h_{j,n},  h_{\ell,n}\right>_\hten^{q-r}
\left<  h_{k,n},  h_{m,n}\right>_\hten^{q-r} \right| 
\le \frac {CM^3}n,
\]
which converges to zero as $n$ tends to infinity, for any fixed $M$. It suffices to handle the sum over one of the sets $D_{i,M}$ for $i=3,4,5$. The analysis is the same for each of them and we consider only the case $i=3$. Set
\[
A_{n,M} =
 \frac {1} {n^2} \sum_{(j,k,\ell,m)\in D_{3,M}}
\left| \left<  h_{j,n},  h_{k,n}\right>_\hten^r 
\left< h_{\ell,n},  h_{m,n}\right>_\hten^r\left< h_{j,n},  h_{\ell,n}\right>_\hten^{q-r}
\left<  h_{k,n},  h_{m,n}\right>_\hten^{q-r} \right| .
\]
  By H\"older's inequality, we can write
\begin{align*}
|A_{n,M} |  & \le   \frac {1} {n^2}
\left(\sum_{(j,k,\ell,m)\in D_{3,M} } \left| \left<  h_{j,n},  h_{k,n}\right>_\hten \right|^q 
 \left| \left< h_{\ell,n},  h_{m,n}\right>_\hten \right|^q
\right) ^{\frac rq} \\
& \qquad \times
 \left( \sum_{(j,k,\ell,m)\in D_{3,M}}   \left| \left< h_{j,n},  h_{\ell,n}\right>_\hten \right|^q 
 \left| \left< h_{k,n},  h_{m,n}\right>_\hten \right|^q 
 \right) ^{1-\frac rq}  \\
 & \le   \left( \frac 1n \sum_{j,k=0 \atop |j-k|>M} ^{\lfloor nt_1 \rfloor -1} \left| \left<  h_{j,n},  h_{k,n}\right>_\hten \right|^q  \right)^{\frac rq}
  \left(  \frac 1n \sum_{\ell ,m=0}^{\lfloor nt_1 \rfloor -1}   \left| \left<  h_{\ell,n},  h_{m,n}\right>_\hten \right|^q  \right)^{2- \frac rq}
\end{align*}
Using the same arguments as in the proof of Lemma \ref{lem5.2}, one can show that the lim sup  as $n$ tends to infinity of the above expression is bounded by
\begin{align*}
 & \left(2^{-q}\sum_{m\in \mathbb{Z}, |m| >M}    \left| |m+1|^\alpha + |m-1| ^{\alpha} -2|m|^\alpha \right| ^q
  \right)^{\frac rq} \\
  &\qquad \times
    \left(2^{-q}\sum_{m\in \mathbb{Z}}    \left| |m+1|^\alpha + |m-1| ^{\alpha} -2|m|^\alpha \right| ^q
  \right)^{2-\frac rq},
  \end{align*}
  which converges to zero as  $M$ tends to infinity. This completes the proof of property (c).

  \medskip
  \noindent
{\it  Proof of condition (d):}  
Condition (d) follows from the proof of Lemma \ref{lem5.2}. In fact, we know that the series
$\sum_{q} c_q^2 q!$ is convergent, and it suffices to take into account the estimates (\ref{eq7}), (\ref{eq12}) and Remark 5.3 after the proof of  Lemma \ref{lem5.2}.

With conditions (a) - (d) of Theorem \ref{thm2} satisfied, it follows that for $i = 1, \dots, p$, $G_n(t_i)$ converges in law to a Gaussian random variable with mean zero and variance given by $\sigma^2(t_i - t_{i-1})$.  We next want to show that $\lim_{n\to\infty} \mathbb{E}\left[ G_n(t_i)G_n(t_j)\right] = 0$ when $t_i \neq t_j$.  Without loss of generality, it is enough to show that
\[ \sup_{q\ge d}\left| \frac{1}{n}\sum_{k=0}^{\lfloor nt_1 \rfloor -1} \sum_{j = \lfloor nt_1 \rfloor}^{\lfloor nt_2 \rfloor -1} \xi_{k,n}^{-q}\xi_{j,n}^{-q}\left( \mathbb{E}\left[ \DXj \Delta X_{\frac kn}\right]\right)^q\right| \]
tends to zero as $n$ tends to infinity.  Let $M_n = (nt_2)^{\frac 13}$, we can decompose the above sum into
\[
\frac{1}{n}\sum_{(j,k)\in D_{1,n}}\xi_{k,n}^{-q}\xi_{j,n}^{-q}\left( \mathbb{E}\left[ \DXj \Delta X_{\frac kn}\right]\right)^q + \frac{1}{n}\sum_{(j,k)\in D_{2,n}}\xi_{k,n}^{-q}\xi_{j,n}^{-q}\left( \mathbb{E}\left[ \DXj \Delta X_{\frac kn}\right]\right)^q,
\]
where
\[
D_{1,n} = \{(j,k)\in D_n: |j-k| \le M_n\}, \quad D_{2,n} = \{ (j,k)\in D_n: |j-k| >M_n\},
\]
and
\[
D_n= \{(j,k): 0\le k \le \lfloor nt_1\rfloor-1, \lfloor nt_1 \rfloor \le j \le \lfloor nt_2 \rfloor-1\}.
\]
By Lemma \ref{lem3.1} and Cauchy-Schwarz, the first sum is bounded by $Cn^{-\frac 13}$.  For the second sum, using arguments from Step 4 and Step 5 in the proof of Lemma \ref{lem5.2} (see \req{eq3}), the  $\limsup$ as $n$ tends to infinity is bounded  by 
\[ C\lim_{n\to\infty} \sum_{m=M_n}^\infty \left( (m+1)^\alpha+(m-1)^\alpha-2m^\alpha\right)^q \le C\lim_{n\to\infty} M_n^{q(\alpha-2)+1} =0,\]
where the limit follows since $M_n =Cn^{\frac 13}$ and $q(\alpha-2)+1 <0$.

Based on the above, Theorem \ref{Vec4MT} implies that the vector sequence $\left( G_n(t_1), \dots, G_n(t_p)\right)$ converges in distribution to  ${\cal N}(0, \mathbf{\Sigma})$ as $n$ tends to infinity, where $\mathbf{\Sigma}$ is the diagonal matrix with entries $\sigma^2(t_i - t_{i-1})$, $i = 1, \dots, p$.  Taking $F_n(t_i) = \sum_{k=1}^i G_n(t_k)$, we have that
\beq{VecFn} \left( F_n(t_1), \dots, F_n(t_p)\right) \law \left( F(t_1), \dots, F(t_p)\right)\eeq
as $n\to\infty$, where each $F(t_i)\sim {\cal N}(0, t_i \sigma^2)$.   This completes the proof of Theorem \ref{thm3}.

%%%%%%%%%%%%%%%%%%%%%%%%%%%%%%%%%%%%%%%%%
 \subsection{Theorem 3.4 examples}
In this section we consider some particular cases where the function $f$ can be represented as a single Hermite polynomial, or a linear combination of finitely many Hermite polynomials.

\medskip
\begin{example}
For integer $p \ge 1$, let $f(x) = x^{2p} - {\mathbb E}\left[ Z^{2p}\right]$, where $Z$ has a ${\cal N}(0,1)$ distribution.  The Stroock formula \cite{Stroock} for a Hermite expansion gives:
\begin{align*}
f(Z) &= \mathbb{E}\left[ f(Z)\right] + \sum_{q=1}^{2p} \frac{1}{q!}\mathbb{E}\left[ f^{(q)}(Z)\right] H_q(Z)\\
&= \sum_{q=1}^{2p} \frac{(2p)!}{q!(2p-q-1)!}\mathbb{E}\left[ Z^{2p-q}\right] H_q(Z).
\end{align*}
Since $2p-1$ is odd, $f$ has Hermite rank $2$,  and Theorem \ref{thm3} can be applied.  In this case, $F_n$ converges in law to a Gaussian random variable with variance $\sigma^2 = \sum_{q=2}^{2p} c_q^2 \sigma_q^2,$ where
\[ c_q = \frac{(2p)!}{q!(2p-q-1)!}\mathbb{E}\left[ Z^{2p-q}\right] = \frac{(2p)!(2p-q-1)!!}{q!(2p-q-1)!}\]
for even $q = 2, \dots, 2p$, and $c_q = 0$ for all odd integers and $q$ greater than $2p$.

For the odd integer case, let $f(x) = |x|^{2p+1} - \mathbb{E}\left[|Z|^{2p+1}\right]$.  Again, since $\mathbb{E}\left[ f'(Z)\right] = 0$, the Hermite rank is 2 and we can apply Theorem \ref{thm3}.  In this case, $\sigma^2 = \sum_{q=2}^{2p} c_q^2 \sigma_q^2,$ where
\[ 
c_q = \frac{(2p+1)!}{q!(2p-q)!}\mathbb{E}\left[ |Z|^{2p-q+1}\right] = \frac{(2p+1)!(2p-q)!!}{q!(2p-q)!}\sqrt{\frac{2}{\pi}}
\]
for even $q = 2, \dots 2p$, and $c_q =0$ for all odd integers $q$ and all $q$ greater than $2p$.
\end{example}

\medskip
\begin{example}
Suppose $f = H_q$ for a single Hermite polynomial of order $q\ge 2$.  In this case, one can show that for fixed $t$, $F_n = F_n(t)$ converges in total variation.  Let $N$ denote a random variable with the ${\cal N}(0,t\sigma_q^2)$ distribution.  In \cite{NoP09} it is proved that
\[ d_{TV}(F_n, N) \le \frac{2}{t\sigma_q^2}\sqrt{\text{Var}\left(\frac 1q\| DF_n\|^2_\hten\right)},\]
where $d_{TV}$ denotes total variation distance and $DF_n$ is the Malliavin derivative of $F_n$ \cite{NoP11, Nualart}.  
We can write  $F_n= I_q( f_{n,q})$, where $f_{n,q}$ is given by  (\ref{eq5}) with $t_1=t$.  From \cite[Lemma 5.2.4]{NoP11}, we have
\begin{equation*}
 \text{Var}\left(\frac 1q\| DF_n\|^2_\hten\right)=\frac{1}{q^2}\sum_{r=1}^{q-1}r^2r!{\binom qr}^4 (2q-2r)!\left\| f_{n,q} \stackrel{\sim}{\otimes}_r f_{n,q} \right\|_{\hten^{2(q-r)}}^2 
\end{equation*}
where $f_{n,q} \stackrel{\sim}{\otimes}_r f_{n,q}$ denotes the symmetrization of $f_{n,q} \otimes_r f_{n,q}$.  Therefore,   by condition (c) in the proof of Theorem \ref{thm3} and using the identity
\[ \left\| f_{n,q} \stackrel{\sim}{\otimes}_r f_{n,q} \right\|_{\hten^{2(q-r)}}^2 \le \left\| f_{n,q} \otimes_r f_{n,q} \right\|_{\hten^{2(q-r)}}^2\] for each $r = 1, \dots, q-1$, we obtain 
\[
\lim_{n\to\infty} d_{TV}(F_n, N) =0.
\]

On the other hand, in this example we can show a functional central limit theorem.  Indeed, by \req{eq15} in the proof of Lemma \ref{lem5.2} and Remark 5.3 following, we can show that for sufficiently large $n$,
\[ 
\mathbb{E}\left[ \left( F_n(t_i) - F_n(t_{i-1})\right)^2\right] \le C\left( \frac{\lfloor nt_i\rfloor - \lfloor nt_{i-1}\rfloor}{n}\right).
\]
Moreover,  using the fact that all the $p$-norms are equivalent on a fixed Wiener chaos,  for arbitrary $0 \le s_1 < s < s_2 \le T$ we deduce that for $n $ large enough
\begin{eqnarray*}
&& \mathbb{E}\left[ \left( F_n(s)-F_n(s_1)\right)^2\left(F_n(s_2)-F_n(s)\right)^2\right]\\
&&\qquad \qquad  \le \left(\mathbb{E}\left[\left(F_n(s) - F_n(s_1)\right)^4\right]\right)^\frac 12 \left(\mathbb{E}\left[\left(F_n(s_2) - F_n(s)\right)^4\right]\right)^\frac 12 \\
&& \qquad \qquad \le C\left( \frac{\lfloor ns_2\rfloor - \lfloor ns_1\rfloor}{n}\right)^2.
\end{eqnarray*}
This implies that the laws  of the processes $\{F_n, n\ge 2\}$ are tight in the Skorohod space ${\bf D} [0,\infty)$ (see Billingsley \cite[Theorem 13.5]{Billingsley}). As a consequence, from Theorem \ref{thm3} we deduce that  the laws of $F_n$ converge in the topology of ${\bf D} [0,\infty)$  to a Brownian motion with scaling $\sigma^2$.
\end{example}

\medskip
%%%%%%%%%%%%%%%%%%%%%%%%%%%%%%%%%%%%%%%%%
\section{Examples of suitable processes}

\subsection{Bifractional Brownian motion}
Bifractional Brownian motion is a generalization of fBm, first introduced by Houdr\'e and Villa in \cite{Houdre}.  It is defined as a centered, Gaussian process $B = \{ B^{H,K}_t, t\ge 0\}$ with covariance
\[{\mathbb E}\left[B_s B_t\right] =\frac{1}{2^K}\left[ (s^{2H} +t^{2H})^K - |t-s|^{2HK}\right],\]
where $H \in (0,1)$ and $K\in(0,1]$.  Note that if $K=1$, then $B$ is an ordinary fBm.  The reader may refer to \cite{Lei, RussoTudor06} for a discussion of properties.  

The covariance can be expressed in terms of $\phi$ with $\beta = HK$ and
\[\phi(x) = \frac{1}{2^K}\left[ (1+x^{2H})^K - (x-1)^{2HK}\right].\]
As stated in Section 1, it is well known that for ordinary fBm, $F_n(t)$ converges in distribution for all $H < 1-\frac{1}{2d}$.   Since the case $K=1$ is well known, we will assume below that $K < 1$.  We now verify the properties (H.1)  -- (H.2).
   
For (H.1), we write
\[ \phi(x) = -\frac{1}{2^K} (x-1)^{2HK} + \frac{1}{2^K}(1+x^{2H})^K,\]
which means we have $\lambda = 2^{-K}$ and $\alpha = 2\beta = 2HK$. Then
$ \psi(x) = 2^{-K}(1+x^{2H})^K$, and the bounds on $|\psi'|$ and $|\psi''|$ follow.   Moreover, $\psi'(1)= HK= \beta \psi(1)$.

For (H.2),
\[ \phi'(x) = 2^{1-K}HK\left( x^{2H-1}(1+x^{2H})^{K-1} - (x-1)^{2HK-1}\right).\]
We can write
\[ (1+x^{2H})^{K-1} = (x^{2H})^{K-1} + (K-1)\int_0^1 (x^{2H}+y)^{K-2} dy,\]
so that
\[ \phi'(x) = 2^{1-K}HK\left( x^{2HK-1}-(x-1)^{2HK-1} +(K-1)x^{2H-1}\int_0^1 (x^{2H}+y)^{K-2} dy\right).\]
Hence, we can write $| \phi'(x)| \le C(x-1)^{2HK-2} + C(x-1)^{2HK-2H-1}$.  If $\alpha < 1$, then we take $\nu = \min\left\{ 1+2H -2HK, 2-2HK\right\} > 1$, and if $\alpha \ge 1$ then $2-2HK \le 2+2H-2HK$ implies $|\phi'(x)| \le C(x-1)^{2HK-2}$.  Continuing, we have
\begin{multline*}
\phi''(x) = 2^{1-K}HK(2HK-1)\left( x^{2HK-2}-(x-1)^{2HK-2}\right)\\ +2^{2-K}H^2K(K-1)x^{2H-2}\int_0^1 (x^{2H}+y)^{K-2}dy\\
+2^{2-K}H^2K(K-1)(K-2)x^{4H-2}\int_0^1 (x^{2H}+y)^{K-3}dy,\end{multline*}
so it follows that
\[ |\phi''(x)| \le C(x-1)^{2HK-3} + C(x-1)^{2HK-2H-2} \le C(x-1)^{-M},\]
where
\[ M=\begin{cases} \min\{2+2H-2HK, 3-2HK\}=\nu+1&\text{ if  }\alpha < 1\\3-2HK&\text{ if  }\alpha \ge 1\end{cases}.\]

\subsection{Subfractional Brownian motion}
Another variant on the fBm is the process known as sub-fractional Brownian motion (sfBm).  This is a centered Gaussian process $\{ S_t, t\ge 0\}$, with covariance defined by:
\beq{Z_cov} R(s,t) = s^{2H} + t^{2H} -\frac 12 \left[ (s+t)^{2H} + |s-t|^{2H}\right],\eeq
with real parameter $H \in (0,1)$.  Some properties of the sfBm are discussed in \cite{Bojdecki, Chavez}.  Note that $H=1/2$ is a standard Brownian motion, and also note the similarity of $R(s,t)$ to the covariance of fBm. Indeed, in \cite{Chavez} it is shown that sfBm may be decomposed into an fBm with Hurst parameter $H$ and another centered Gaussian process. 

Let $S = \{ S_t, t\ge 0\}$ denote a sub-fractional Brownian motion with $0< H < 1$.  In this case we have $\lambda =\frac 12$, $\alpha = 2\beta = 2H$ and
\[\phi(x) = 1+x^{2H} -\frac 12\left( (x+1)^{2H} + (x-1)^{2H}\right)=-\lambda (x-1)^{2H} + \psi(x),
\]
where $\psi(x) = 1+x^{2H} -\frac 12(x+1)^{2H}$.  Clearly, $|\psi'(x)| \le Cx^{2H-1}$ and $|\psi''(x)| \le Cx^{2H-2}$.   Moreover, $\psi'(1)= H(2-2^{2H-1}) =\beta \psi(1)$.

To check (H.2), we have
\[ |\phi'(x)| = H\left| (x+1)^{2H-1} -2x^{2H-1} +(x-1)^{2H-1}\right|\le C(x-1)^{2H-2},\]
so (H.2)(i) is satisfied for $x\ge 2$, and (H.2)(ii) can be shown by a second derivative.

%%%%%%%%%%%%%%%%%%%%%%%%%%%%%%%Fragment - Swanson Process
\subsection{A Gaussian process introduced by J. Swanson}
We consider the centered Gaussian process $Y = \{Y_t, t\ge 0\}$ with covariance given by
\[{\mathbb E}\left[ Y_sY_t\right] = \sqrt{st}~\sin^{-1}\left(\frac{s\wedge t}{\sqrt{st}}\right).\]
Then $Y$ can be characterized as a self-similar Gaussian process with $\beta = 1/2$ and \[\phi(x) = \sqrt x \sin^{-1}\left(\frac{1}{\sqrt x}\right).\]
This process was studied by Jason Swanson in a 2007 paper \cite{Swanson07}, and it arises as the limit of normalized empirical quantiles of a system of independent Brownian motions.  The properties of this process were also considered in \cite{HN_14}.
Unlike the fBm family processes in Sections 4.1 and 4.2, this process is an example of the case $\alpha < 2\beta$. 

This process satisfies (H.1), with $\beta=1/2$ and $\alpha=1/2$.  We can write this as
\[ \phi(x) = -\sqrt{x-1} + \left( \sqrt x \sin^{-1}\left(\frac{1}{\sqrt x}\right) + \sqrt{x-1}\right),\]
which gives $\lambda = 1$ and 
\[ \psi(x) =  \sqrt x \sin^{-1}\left(\frac{1}{\sqrt x}\right) + \sqrt{x-1}.\]
Then
\[ | \psi'(x)| = \frac{x^{-\frac 12}}{2}\left| \sin^{-1}\left(\frac{1}{\sqrt x}\right) - \frac{\sqrt x -1}{\sqrt{x-1}}\right| \le Cx^{-\frac 12}.\]
And for the second derivative,
\[
\psi''(x) = -\frac{x^{-\frac 32}}{4}\left(\sin^{-1}\left(\frac{1}{\sqrt x}\right) - \frac{\sqrt x -1}{\sqrt{x-1}}\right) -\frac{x^{-\frac 12}}{4\sqrt{x-1}}\left( \frac{1}{\sqrt x(\sqrt x +1)}\right),\]
so $|\psi''(x) |\le Cx^{-\frac 32}\left( 1+(x-1)^{-\frac 12}\right) \le Cx^{-1}(x-1)^{-\frac 12}$.

To check condition (H.2), we return to the original expression $\phi(x) = \sqrt x \sin^{-1}\left( x^{-\frac 12}\right)$.  We have
\begin{align*}
\phi'(x) &=  \frac{1}{2\sqrt x} \left( \sin^{-1}\left(\frac{1}{\sqrt x}\right) - \frac{1}{\sqrt{x-1}}\right)\\
&=  \frac{1}{2\sqrt x} \left( \frac{1}{\sqrt x} - \frac{1}{\sqrt{x-1}}\right)+\frac{1}{2\sqrt x} \left( \sin^{-1}\left(\frac{1}{\sqrt x}\right) - \frac{1}{\sqrt{x}}\right).\end{align*}
From a Taylor expansion of $\sin t$, we have for $0< t < 1$
\[ \sin t = t - \frac{t^3}{3!} + \frac{h^5}{5!},\]
for some $0 \le h \le t$, and it follows that
\[ t = \sin^{-1}\left( t - \frac{t^3}{3!} + \frac{h^5}{5!}\right).\]
We then  set $t = x^{-\frac 12},$ and use a second Taylor expansion on $\sin^{-1}$, to conclude that for $x>2$,
\[ \left| \frac{1}{\sqrt x} - \sin^{-1}\left(\frac{1}{\sqrt x}\right) \right| \le \frac{C}{x\sqrt{x-1}}.\]
Hence, we have
\[ | \phi'(x)| \le \frac{1}{2\sqrt x}\left| \frac{1}{\sqrt x} - \frac{1}{\sqrt{x-1}}\right| + \frac{C}{x^{\frac 32}\sqrt{x-1}} \le C(x-1)^{-2},\]
so we have $\nu = 2$ since $\alpha = \frac 12 < 1$.

Similarly,
\begin{align*}
\phi''(x) &= -\frac 14 x^{-\frac 32}\left( \sin^{-1}\left( \frac{1}{\sqrt x}\right) - \frac{1}{\sqrt{x-1}}\right) - \frac 14 x^{-\frac 12}\left( \frac{1}{x\sqrt{x-1}} - (x-1)^{-\frac 32}\right)\\
&=-\frac 14 x^{-\frac 32}\left( \sin^{-1}\left( \frac{1}{\sqrt x}\right) - \frac{1}{\sqrt{x-1}}\right) - \frac{1}{4\sqrt{x(x-1)}}\left( \frac 1x - \frac{1}{x-1}\right),
\end{align*} 
Hence, $| \phi''(x) | \le C(x-1)^{-3} = C(x-1)^{-\nu -1}$ for $x > 2$.

\subsection{Two smooth processes with $\alpha < 1$}
For $0 < \alpha < 1$, we consider the centered Gaussian processes $Z_1(t),~Z_2(t)$, with covariances given by:
\begin{align*}
\mathbb{E}\left[ Z_1(s)Z_1(t)\right] &= \Gamma(1-\alpha)\left( (s+t)^\alpha - \max(s,t)^\alpha\right)\\
\mathbb{E}\left[ Z_2(s)Z_2(t)\right] &= \Gamma(1-\alpha)\left( s^\alpha+t^\alpha - (s+t)^\alpha\right).
\end{align*}
These processes are discussed in a recent paper by Durieu and Wang \cite{DuWa}, where it is shown that the process $Z = Z_1 + Z_2$ (where $Z_1,~Z_2$ are independent) is the limit in law of a discrete process studied by Karlin.  The process $Z_1$ is new, but the process $Z_2$, with a different scaling constant, was first described in Lei and Nualart \cite{Lei}.

For $Z_1$, we can write $\mathbb{E}\left[ Z_1(s)Z_1(t)\right] = s^\alpha \phi_1(t/s)$, 
where $\phi_1(x) = \Gamma(1-\alpha)\left((x+1)^\alpha - x^\alpha\right)$.  We can also write $\phi_1$ in the form
\[ \phi_1(x) = -\Gamma(1-\alpha)(x-1)^\alpha + \Gamma(1-\alpha)\left((x+1)^\alpha+(x-1)^\alpha -x^\alpha\right),\]
so that $\lambda = \Gamma(1-\alpha)$ and $\psi_1(x) = \Gamma(1-\alpha)\left((x+1)^\alpha+(x-1)^\alpha-x^\alpha\right)$.  Note that by a Taylor expansion for $x \ge 1$,
\beq{Taylor_x-1} (x-1)^\alpha = x^\alpha -\alpha x^{\alpha-1} + \frac{\alpha(\alpha-1)}{2}x^{\alpha-2} + O(x^{\alpha-3}),\eeq
so that we have $|\psi_1'(x)| \le Cx^{\alpha-1}$ and $|\psi_2''(x)| \le Cx^{\alpha-2}$, and $Z_1$ satisfies (H.1).  For (H.2), note that we can write
\[ \phi_1(x) = \alpha\Gamma(1-\alpha)\int_0^1 (x+u)^{\alpha-1}du,\]
so for $x\geq 2$ 
\[ |\phi_1'(x)| = \alpha|\alpha-1|\Gamma(1-\alpha)\int_0^1(x+u)^{\alpha-2}du \le 2x^{\alpha-2},\]
which satisfies (H.2)(i) with $\nu = 2-\alpha >1$; and it can also be seen that $|\phi_1''(x)| \leq 4x^{\alpha-3} = 2x^{-\nu-1}$, so that (H.2)(ii) is satisfied.

\medskip
For $Z_2$ we have
\begin{align*} \phi_2(x) &= \Gamma(1-\alpha)(1 + x^\alpha - (x+1)^\alpha)\\
&=-\Gamma(1-\alpha)(x-1)^\alpha + \Gamma(1-\alpha)\left( 1+x^\alpha +(x-1)^\alpha - (x+1)^\alpha\right)\\ &= -\Gamma(1-\alpha) (x-1)^\alpha + \psi_2(x),
\end{align*}
where again we take $\lambda = \Gamma(1-\alpha)$.  By a computation similar to that for $Z_1$ above, it can be seen that $\psi_2$ satisfies (i) and (ii) of (H.1).  The computations for (H.2) are also similar to those for $Z_1$ above.  We write
\[ \phi_2(x) = \Gamma(1-\alpha) - \alpha\Gamma(1-\alpha)\int_0^1 (x+u)^{\alpha-1}du,\]
so that (H.2) conditions (i) and (ii) are satisfied with $\nu = 2-\alpha >1$.

\begin{remark}  Both processes satisfy $\mathbb{E}\left[ Z_i(t)^2\right] = Ct^\alpha$, so the increment exponent at 0 is $\alpha$.  On the other hand, for $t>0$, $\mathbb{E}\left[ (Z_i(t+s) - Z_i(t))^2\right] \approx c_t s$ for a constant $c_t$, and the increment exponent for $t >0$ is 1.  The renormalization in Theorem \ref{thm3} is given by $\alpha$, which is also the self-similarity parameter.
\end{remark}

%%%%%%%%%%%%%%%%%%%%%%%%%%%%%%%%%%%%%%%%%

%%%%%%%%%%%%%%%%%%%%%%%%%%%%%%%%Fragment - re-worked Tudor Process
%\subsection{A Gaussian process introduced by C.A. Tudor}

%%%%%%%%%%%%%%%%%%%%%%%%%%%%%%%%%%

\section{Some technical lemmas}

We begin with a technical lemma that gives upper bounds on certain covariance terms.  

\begin{lemma}  \label{lem5.1}
Let $n \ge 6$ be an integer, and let $j,k\ge 1$ be integers satisfying $3k \le j$.  Then under (H.2), there is a constant $C>0$ such that
%Let $t>0$, let $j,k \ge 0$ be integers and let $n\ge 6$  be an integer.  Under (H.1) -- (H.2), we have the following estimates, up to a positive constant $C$:
%\begin{enumerate}[(a)]
%\item For $0\le j\le \Nt$,
% \[\sup_{0\le j \le \Nt} {\mathbb E}\left[ \Delta X_{\frac jn}^2\right] \le Cn^{-\alpha}.\]
%\item  For $2\le j \le \Nt-1$,
%\[ \left| {\mathbb E}\left[ \Delta X_{\frac jn} \Delta X_{\frac 0n}\right] \right| \le \begin{cases} Cn^{-2\beta}j^{-\nu}& \text{if }\alpha < 1\\ Cn^{-2\beta}j^{\alpha-2}&\text{if }\alpha \ge 1
%\end{cases},\]
%where $\nu > 1$ is defined in (H.2).
%\item  For $1\le k \le j-2\le 3k$,
%\[ \left| {\mathbb E}\left[ \Delta X_{\frac jn} \Delta X_{\frac{k}{n}}\right] \right| \le Cn^{-2\beta} k^{2\beta-\alpha-1}(j-k)^{\alpha-1} + Cn^{-2\beta}k^{2\beta -\alpha}(j-k)^{\alpha -2}.\]
%\item  For $1\le k$ and  $j \ge  3k$,
\[ \left| {\mathbb E}\left[ \Delta X_{\frac jn} \Delta X_{\frac{k}{n}}\right] \right| \le \begin{cases} Cn^{-2\beta}k^{2\beta +\nu-2}(j-k)^{-\nu}&\text{ if  } \alpha < 1\\Cn^{-2\beta}k^{2\beta -\alpha}(j-k)^{\alpha-2}&\text{ if  } \alpha\ge 1\end{cases},\]
where the exponent $1 < \nu \le 2$ is defined in (H.2). 
%\end{enumerate}
\end{lemma}

\begin{proof}
%For (a), first note that if $j=0$, then 
%\[ {\mathbb E}\left[ \Delta X_{\frac 0n}^2\right] = {\mathbb E}\left[ X_{\frac 1n}^2\right] = n^{-2\beta}\phi(1) \le Cn^{-\alpha}.\]
%For $j >0$, the result follows from Lemma \ref{lem3.1}.

%\medskip
%For (b), we have
%\[ {\mathbb E}\left[ \Delta X_{\frac jn} \Delta X_{\frac 0n}\right] = {\mathbb E}\left[ X_{\frac{j+1}{n}}X_{\frac 1n} - X_{\frac jn}X_{\frac 1n}\right] = n^{-2\beta}\left( \phi(j+1) - \phi(j)\right).\]
%Since $\left| \phi(j+1) - \phi(j) \right| \le C\phi'(j)$ for $j\ge 2$, the result follows from (H.2).

%\medskip
%Part (b) is a direct result of Lemma \ref{lem3.2}(b).

\medskip
We have
\begin{align*}
{\mathbb E}\left[ \Delta X_{\frac jn} \Delta X_{\frac{k}{n}}\right] &= n^{-2\beta}(k+1)^{2\beta}\left( \phi\left( \frac{j+1}{k+1}\right) - \phi\left( \frac{j}{k+1}\right)\right)\\
&\qquad\quad - n^{-2\beta}k^{2\beta}\left( \phi\left( \frac{j+1}{k}\right) - \phi\left( \frac{j}{k}\right)\right)\\
&= n^{-2\beta}\left((k+1)^{2\beta}-k^{2\beta}\right)\left( \phi\left( \frac{j+1}{k+1}\right) - \phi\left( \frac{j}{k+1}\right)\right)\\
&\qquad\quad +n^{-2\beta}k^{2\beta}\left[ \phi\left( \frac{j+1}{k+1}\right) - \phi\left( \frac{j}{k+1}\right) - \phi\left( \frac{j+1}{k}\right) + \phi\left( \frac{j}{k}\right)\right].
\end{align*}
Condition $j \ge 3k$ and (H.2) imply that there exists a constant $C>0$ such that for each $x\in\left [\frac j{k+1}, \frac {j+1} {k+1}\right]$
\[ 
\left| \phi' \left( x \right)\right| \le \begin{cases} Ck^{\nu}(j-k)^{-\nu}&\text{ if  } \alpha < 1\\Ck^{2-\alpha}(j-k)^{\alpha-2}&\text{ if  }\alpha \ge 1\end{cases}
\]
and  for each $x\in \left[ \frac j {k+1}, \frac {j+1} k \right]$,
\[ 
\left| \phi''\left( x \right)\right| \le \begin{cases} Ck^{\nu+1}(j-k)^{-\nu-1}&\text{ if  } \alpha < 1\\Ck^{3-\alpha}(j-k)^{\alpha-3}&\text{ if  }\alpha \ge 1\end{cases}.
\]
Then the desired estimate follows easily from the Mean Value Theorem.
\end{proof}

\begin{lemma} \label{lem5.2} Let $t>0$.  Under above definitions with $\alpha < 2-\frac 1q$, 
\[ 
\lim_{n\to\infty} {\mathbb E}\left[F_{n,q}^2(t)\right]=\sigma_q^2,
\]
 where $\sigma_q^2$ is given by \req{var_main}.
 \end{lemma}

\begin{proof}
The proof will be done in several steps.

\noindent{\it Step 1.} 
It follows from (\ref{Herm_cov}) that
\begin{align*}
{\mathbb E}\left[F_{n,q}^2(t)\right] &= \frac {c_q^2} {n}\sum_{j,k=0}^{\Nt -1} {\mathbb E}\left[ H_q \left(  Y_{j,n}\right)H_q \left( Y_{k,n}\right)\right]\\
&= \frac {q!c_q^2} {n}\sum_{j,k = 0}^{\Nt -1}  \| \Delta X_{\frac jn}\|^{-q}_{L^2(\Omega)} 
 \| \Delta X_{\frac kn}\|^{-q}_{L^2(\Omega)} 
\left( {\mathbb E}\left[ \DXj \Delta X_{\frac kn}\right]\right)^q.
\end{align*}
Fix $\gamma \in (0,1/2)$ and decompose the above double sum into two terms, that is,
${\mathbb E}\left[F_{n,q}^2\right]  =A_{1,n} + A_{2,n}$, where
\[
A_{i,n}  =\frac { q!c_q^2} {n} \sum_{j,k \in D_i}  \| \Delta X_{\frac jn}\|^{-q}_{L^2(\Omega)} 
 \| \Delta X_{\frac kn}\|^{-q}_{L^2(\Omega)} 
\left( {\mathbb E}\left[ \DXj \Delta X_{\frac kn}\right]\right)^q,
\]
and $D_1=\{ j :0\le j \le n^\gamma \wedge  (\Nt -1)\}$ and
$D_2=\{ j :n^\gamma <  j \le\Nt -1 \}$. 

The term $A_{1,n}$ can be bounded, using Cauchy-Schwarz inequality, by  $q! c_q^2n^{ 2\gamma -1}$ and it converges to zero as $n$ tends to infinity. So, it suffices to consider the term $A_{2,n}$. We recall the notation   $ \xi_{j,n}= \| \Delta X_{\frac jn}\|_{L^2(\Omega)} $.

\medskip
\noindent {\it Step 2.}  From (\ref{eq1}) we can write for $j\in D_2$
\begin{equation}  \label{eq5}
\xi_{j,n} ^2=2\lambda j^{2\beta-\alpha} n^{-2\beta} \left[ 1+ \eta_{j,n}\right],
\end{equation}
where
\begin{equation}  \label{eq4}
|\eta_{j,n}|= \frac 1{2\lambda} j^{\alpha-2\beta} n^{ -2\beta} \left|g_{1}\left( \frac jn, \frac 1n\right)\right|
\le Cn^{ -\gamma \varepsilon}.
\end{equation}
Consider the following decomposition:
\begin{align*}
 & A_{2,n} = \frac { q!c_q^2} {n} \sum_{j,k \in D_2} \xi_{j,n} ^{-q}  \xi_{k,n} ^{-q}\left( {\mathbb E}\left[ \Delta X_{\frac jn}\Delta X_{\frac kn}\right]\right)^q\\
&\qquad= q!  c_q^2 \frac { \Nt-1-n^{\gamma}}{n}  + \frac { q!c_q^2} {n}\sum_{j ,k\in D_2, |k-j|=1}  \xi_{j,n} ^{-q}  \xi_{k,n} ^{-q}  \left( {\mathbb E}\left[ \Delta X_{\frac jn}\Delta X_{\frac{k}{n}}\right]\right)^q \\
&\qquad +\frac { q!c_q^2} {n} \sum_{ j,k\in D_2,  |k-j|\ge 2}  \xi_{j,n} ^{-q}  \xi_{k,n} ^{-q}  \left( {\mathbb E}\left[ \Delta X_{\frac jn}\Delta X_{\frac kn}\right]\right)^q.
\end{align*}
The first term clearly converges to $q! c_q^2t$. We denote the second and third term by $B_{2,n}$ and $B_{3,n}$, respectively.

\medskip
\noindent {\it Step 3.}  Let us consider the term  $B_{2,n}$. 
Using  Lemma \ref{lem3.2}(a) we can write
\[
{\mathbb E}\left[ \Delta X_{\frac jn}\Delta X_{\frac{j-1}{n}}\right]
 = \lambda (2^\alpha -2) j^{2\beta-\alpha} n^{-2\beta} +  g_2\left(\frac jn, \frac 1n\right).
 \]
 Multiplying this expression by $\xi_{j,n} ^{-1} \xi_{j-1,n} ^{-1}$ and using (\ref{eq5}) yields
\begin{align}  \notag
 \xi_{j,n} ^{-1} \xi_{j-1,n} ^{-1}{\mathbb E}\left[ \Delta X_{\frac jn}\Delta X_{\frac{j-1}{n}}\right]
& = [1+ \eta_{j,n}]^{-\frac 12}  [1+ \eta_{j-1,n}]^{-\frac 12}  \\
&  \qquad   \times 
\left( (2^{\alpha-1} -1) \left(\frac j{j-1} \right)^{\frac 12(2\beta-\alpha)} 
+ R_{j,n}   \right),  \label{eq8} 
\end{align}
where
\[
R_{j,n}= (2\lambda)^{-1} n^{2\beta} (j(j-1))^{-\frac 12(2\beta-\alpha)}   g_2\left(\frac jn, \frac 1n\right).
\]
Applying Lemma \ref{lem3.2}(a) and assuming $n^\gamma \ge 2$,  this term can be bounded as follows.
\begin{equation} \label{eq6}
|R_{j,n}| \le C   j^{-(2\beta-\alpha +\delta)},   
\end{equation}
where $\delta = \min( 2-2\beta, 1+\alpha - 2\beta)$.  We claim that there exist $\rho \in (0,1)$ such that for $n$ large enough and for all  $j\in D_2$,
\begin{equation} \label{eq7}
\left| \xi_{j,n} ^{-1} \xi_{j-1,n} ^{-1}{\mathbb E}\left[ \Delta X_{\frac jn}\Delta X_{\frac{j-1}{n}}\right]
 \right| <\rho <1.
 \end{equation}
 This follows from the estimates (\ref{eq4}),  (\ref{eq6}) and the fact that  $|2^{\alpha-1} -1|<1$ and 
$2\beta-\alpha +\delta>0$. Finally,  from the expression (\ref{eq8}) and the estimates (\ref{eq4}) and (\ref{eq6})  it follows that
\begin{align}  \notag
\lim_{n\to \infty}  B_{2,n}&= q! c_q^2 2
\lim_{n\to \infty}   \frac 1n   \sum_{j,j-1 \in D_2}
[1+ \eta_{j,n}]^{-\frac q2}  [1+ \eta_{j-1,n}]^{-\frac q2}   \\
&\qquad \times   \notag
\left( (2^{\alpha-1} -1) \left(\frac j{j-1} \right)^{\frac 12(2\beta-\alpha)} 
+ R_{j,n}   \right)^q \\   \label{eq9}
&=q!  c_q^2 2 (2^{\alpha -1} -1)^q  t.
\end{align}

\medskip
\noindent {\it Step 4.}  Let us consider the term  $B_{3,n}$.   First we show that the terms with
$ k\le  \lfloor j/3\rfloor $ or  $ j\le  \lfloor k/3\rfloor $   do not contribute to the limit.  That is, we claim that the following expression converges to zero as $n$ tends to infinity
\[
C_{n}= \frac 1n \sum_{j=3}^{\Nt -1}\sum_{k=1}^{\left\lfloor \frac j3 \right\rfloor}  \xi_{j,n}^{-q} \xi_{k,n}^{-q}
\left|{\mathbb E}\left[ \Delta X_{\frac jn}\Delta X_{\frac kn}\right]\right|^q.
\]
To do this, we consider two cases. 

{\it Case 1.}
  When $\alpha < 1$, Lemma \ref{lem5.1} gives
\[
C_n \le  Cn^{-1}\sum_{j=3}^{\Nt -1}\sum_{k=1}^{\left\lfloor \frac j3\right\rfloor} 
k^{q(\alpha+\nu-2)}(j-k)^{-q\nu}.
\]
%If $2\beta +\nu-2 < 0$, then $k^{q(2\beta+\nu-2)} \le 1$, and note that $1\le k \le j/3$ implies $(j-k)^{-\nu} \le Cj^{-\nu}$, where $C$ is a constant that does not depend on $j$ or $k$.  Hence, we can write
%\[
%C_n \le C \frac 1n \sum_{j=3}^{\Nt -1}j^{-q(2\beta-\alpha)},
%\]
%which converges to zero as $n$ tends to infinity. 
%On the other hand, if $2\beta +\nu-2 \ge 0$, then 
Note that $1 \le k \le j/3$ implies $(j-k)^{-\nu} \le Cj^{-\nu}$, so it follows that
\[
C_n \le  \frac Cn \sum_{j=3}^{\Nt -1}j^{-q\nu}\sum_{k=1}^{\left\lfloor \frac j3 \right\rfloor} k^{q(\alpha +\nu-2)} \le  \frac Cn \sum_{j=3}^{\Nt -1}j^{q(\alpha-2)+1}
\]
which converges to zero as $n$ tends to infinity because 
$q(\alpha-2)+1<0$ since $\alpha < 2-\frac 1d$ and $q\ge d$.

{\em Case 2:  Assume $1 \le \alpha < 2-\frac 1q$}.  For this case, by the other part of Lemma \ref{lem5.1},
\[
C_n \le  \frac Cn \sum_{j=3}^{\Nt -1}\sum_{k=1}^{\left\lfloor \frac j3\right\rfloor} 
(j-k)^{q(\alpha -2)}.
\]
Again, since $1\le k \le j/3$, we can say that $(j-k)^{\alpha -2} \le Cj^{\alpha -2}$, where $C$ is a constant that does not depend on $j$ or $k$.  Thus, we have
\[
C_n \le    \frac Cn \sum_{j=3}^{\Nt -1}j^{q(\alpha-2)+1},
\]
which converges to zero as $n$ tends to infinity because
$q(\alpha-2)+1<0$ since $\alpha < 2-\frac 1d$ and $q\ge d$.

 \medskip
\noindent {\it Step 5.} 
Finally, it remains to study the following term 
\[
 D_n=\frac {2q!c_q^2} n  \sum_{ j,k \in D_2, \lfloor j/3\rfloor   <k \le j-2 }
 \xi_{j,n}^{-q} \xi_{k,n}^{-q}
\left({\mathbb E}\left[ \Delta X_{\frac jn}\Delta X_{\frac kn}\right]\right)^q.
\]
We have by Lemma \ref{lem3.2}(b),
\[
{\mathbb E}\left[ \Delta X_{\frac jn}\Delta X_{\frac kn}\right]
=\lambda n^{-2\beta}k^{2\beta-\alpha}A_{j,k} + g_3\left(\frac {k+1}n, \frac {j+1}n, \frac 1n\right),
\]
where
\[ 
A_{j,k} =   |j-k+1|^\alpha +|j-k-1|^\alpha- 2|j-k|^\alpha.
\]
Multiplying this expression by  $\xi_{j,n}^{-1} \xi_{k,n}^{-1} $ and using (\ref{eq5}) yields
\begin{align*}
\xi_{j,n}^{-1} \xi_{k,n}^{-1} \left[ \lambda n^{-2\beta}k^{2\beta -\alpha}A_{j,k} + g_3\left( \frac {k+1}n, \frac {j+1}n, \frac 1n\right)\right]
&=[1+ \eta_{j,n}]^{-\frac 12}  [1+ \eta_{k,n}]^{-\frac 12}  \\
& \quad \times
\left(  2^{-1} \left(  k/j \right)^{\frac 12(2\beta -\alpha)}   A_{j,k} +  R_{j,k,n} \right),
\end{align*}
where
\[
R_{j,k,n}=(2\lambda)^{-1} n^{2\beta} (kj)^{-\frac 12 (2\beta-\alpha)} g_3\left( \frac {k+1}n, \frac { j+1}n, \frac 1n\right).
\]
We can write
\begin{align*}
\left(  2^{-1}(k/j)^{\frac 12(2\beta -\alpha)}  A_{j,k} +  R_{j,k,n} \right)^q &=
2^{-q} (k/j)^{\frac q2(2\beta -\alpha)}    A^q_{j,k} \\
&\qquad +\sum_{r=1}^q \binom{q}r 2^{-(q-r)} (k/j)^{\frac {q-r}2(2\beta -\alpha)}   A^{q-r}_{j,k}  R_{j,k,n}^r\\
& =: \Phi_{1,j,k,n} + \Phi_{2,j,k,n}.
\end{align*}
From the estimate for $|g_3|$ from Lemma \ref{lem3.2}(b), we have
\[
| R_{j,k,n}|  \le  C j^{-\frac 12(2\beta-\alpha)}   k^{\frac 12 (2\beta-\alpha)-1} (j-k-1)^{\alpha-1}  
 + Cj^{-\frac 12(2\beta-\alpha)} k^{\frac 12 (2\beta-\alpha)+(\alpha -2)}.\]
It follows that we  can write 
\[
|\Phi_{2,j,k,n}| \le  \sum_{r=1}^q \binom{q}r  C^r  2^{-(q-r)} A_{j,k} ^{q-r} (k/j)^{\frac  q2(2\beta -\alpha)}       \left( k^{-r} (j-k-1)^{r(\alpha-1)}+k^{r(\alpha-2)}\right).
\]
Using that $k/j \le 1$,    $\left| A_{j,k}\right| \le C|j-k-1|^{\alpha -2}$ for $|j-k| \ge 2$,  
$k^{-r} \le C ^rj^{-r} \le  C^r(j-k-1)^{-r}$, and similarly that $k^{\alpha-2}\le C(j-k-1)^{\alpha-2}$, we obtain
\[
|\Phi_{2,j,k,n}| \le C^q            (j-k-1)^{q(\alpha-2)}
\]
for some constant $C>1$.
This implies
\[
\frac {2q!c_q^2} n \sum_{j=\lfloor n^\gamma \rfloor}^{\Nt -1}\sum_{k=\lfloor j/3\rfloor +1}^{j-2} [1+ \eta_{j,n}]^{-\frac q2}  [1+ \eta_{k,n}]^{-\frac q2}  |\Phi_{2,j,k,n}| \le
\frac {2q!c_q^2} n C^q \sum_{j=\lfloor n^\gamma \rfloor}^{\Nt -1}      j^{q(\alpha-2)+1}.
\]
which converges to zero as $n$ tends to infinity. Moreover,  there exists $\rho\in (0,1)$ such that
\begin{equation} \label{eq12}
\frac 1 n C^q \sum_{j=\lfloor n^\gamma \rfloor}^{\Nt -1}      j^{q(\alpha-2)+1} \le \rho
\end{equation}
for $n$ large enough.

 Therefore,
\[
\lim_{n\rightarrow \infty} D_n= q! c_q^2
\lim_{n\rightarrow \infty}  \frac 2n 
 \sum_{  j,k \in D_2, \lfloor j/3\rfloor   <k \le j-2 }
 2^{-q} (k/j)^{\frac q2(2\beta -\alpha)}    A_{j,k}^q.
  \]
Since we know from Step 1 and   Step 4 that the terms with $j,k\in D_1$ and  $ k\le  \lfloor j/3\rfloor $ do not contribute to the limit, we can add these terms and write
\[
\lim_{n\rightarrow \infty} D_n= q! c_q^2
\lim_{n\rightarrow \infty}  \frac 2n 
\sum_{j=3}^{\Nt -1}\sum_{k= 1}^{j-2}
 2^{-q} (k/j)^{\frac q2(2\beta -\alpha)}   A_{j,k}^q.
  \]
We use the change of index $m = j-k$ to write
\begin{equation}   \label{eq15}
\lim_{n\rightarrow \infty} D_n= q! c_q^2
\lim_{n\rightarrow \infty}  \frac 2n 
\sum_{k=1}^{\Nt -1} \sum_{m=2}^{\Nt-k-1}
 2^{-q} \left (\frac k{k+m} \right)^{\frac q2(2\beta -\alpha)} \left(   (m+1)^\alpha +(m-1)^\alpha- 2m^\alpha \right)^q,
  \end{equation} 
which leads to
\begin{equation}   \label{eq3}
\lim_{n\rightarrow \infty} D_n= 2^{1-q} q!  c_q^2t 
\sum_{m=2}^\infty \left(   (m+1)^\alpha +(m-1)^\alpha- 2m^\alpha \right)^q.
\end{equation}
Thus, the limit \req{var_main} follows from results \req{eq9} and \req{eq3}, and Lemma \ref{lem5.2} is proved.
%Notice also that the right-hand side of (\ref{eq15}) is bounded by
%\[
%2q! c_q^2
%T\sum_{m=2}^{\infty}
% 2^{-d}  \left(   (m+1)^\alpha +(m-1)^\alpha- 2m^\alpha \right)^d  \rho_m^{q-d},
% \]
% where $\sup_{m\ge 2}\rho_m <1$.

\end{proof}
 
\begin{remark}  Notice  that the right-hand side of (\ref{eq15}) is bounded by
\[
2q! c_q^2
t\sum_{m=2}^{\infty}
 2^{-d}  \left(   (m+1)^\alpha +(m-1)^\alpha- 2m^\alpha \right)^d  \rho_m^{q-d},
 \]
 where $\sup_{m\ge 2}\rho_m <1$.
\end{remark}

\end{document}